\documentclass[oneside,english]{amsart}
\usepackage[T1]{fontenc}
\usepackage[latin9]{inputenc}
\usepackage{color}
\usepackage{babel}
\usepackage{amstext}
\usepackage{amsthm}
\usepackage{yhmath}
\usepackage{amssymb}
\usepackage{graphicx}
\usepackage{subcaption}
\usepackage{tikz}
\usetikzlibrary{calc,intersections,through,backgrounds,decorations.pathreplacing,calligraphy}
\usetikzlibrary{math}
\usepackage[unicode=true,pdfusetitle,
 bookmarks=true,bookmarksnumbered=false,bookmarksopen=false,
 breaklinks=false,pdfborder={0 0 0},pdfborderstyle={},backref=false,colorlinks=false]
 {hyperref}

\makeatletter
\numberwithin{equation}{section}
\numberwithin{figure}{section}
\theoremstyle{plain}
\newtheorem{thm}{\protect\theoremname}[section]
\theoremstyle{remark}
\newtheorem{rem}[thm]{\protect\remarkname}
\theoremstyle{plain}
\newtheorem{prop}[thm]{\protect\propositionname}
\theoremstyle{plain}
\newtheorem{lem}[thm]{\protect\lemmaname}
\theoremstyle{definition}
\newtheorem{defn}[thm]{\protect\definitionname}

\makeatother

\DeclareMathOperator{\dist}{dist}

\providecommand{\definitionname}{Definition}
\providecommand{\lemmaname}{Lemma}
\providecommand{\propositionname}{Proposition}
\providecommand{\remarkname}{Remark}
\providecommand{\theoremname}{Theorem}
\begin{document}
\global\long\def\R{\mathbf{\mathbb{R}}}%
\global\long\def\C{\mathbf{\mathbb{C}}}%
\global\long\def\Z{\mathbf{\mathbb{Z}}}%
\global\long\def\N{\mathbf{\mathbb{N}}}%
\global\long\def\T{\mathbb{T}}%
\global\long\def\Im{\mathrm{Im}}%
\global\long\def\Re{\mathrm{Re}}%
\global\long\def\eval{\mathrm{eval}}%
\global\long\def\irr{\mathrm{irr}}%
\global\long\def\H{\mathcal{H}}%
\global\long\def\M{\mathbb{M}}%
\global\long\def\P{\mathcal{P}}%
\global\long\def\L{\mathcal{L}}%
\global\long\def\F{\mathcal{\mathcal{F}}}%
\global\long\def\s{\sigma}%
\global\long\def\Rc{\mathcal{R}}%
\global\long\def\W{\tilde{W}}%
\global\long\def\avg{\mathrm{avg}}%

\global\long\def\G{\mathcal{G}}%
\global\long\def\d{\partial}%
 
\global\long\def\jp#1{\langle#1\rangle}%
\global\long\def\norm#1{\|#1\|}%
\global\long\def\mc#1{\mathcal{\mathcal{#1}}}%

\global\long\def\Right{\Rightarrow}%
\global\long\def\Left{\Leftarrow}%
\global\long\def\les{\lesssim}%
\global\long\def\hook{\hookrightarrow}%

\global\long\def\D{\mathbf{D}}%
\global\long\def\rad{\mathrm{rad}}%

\global\long\def\env{\mathrm{Env}}%
\global\long\def\re{\mathrm{re}}%
\global\long\def\im{\mathrm{im}}%
\global\long\def\err{\mathrm{Err}}%

\global\long\def\d{\partial}%
 
\global\long\def\jp#1{\langle#1\rangle}%
\global\long\def\norm#1{\|#1\|}%
\global\long\def\ol#1{\overline{#1}}%
\global\long\def\wt#1{\widehat{#1}}%
\global\long\def\tilde#1{\widetilde{#1}}%

\global\long\def\br#1{(#1)}%
\global\long\def\Bb#1{\Big(#1\Big)}%
\global\long\def\bb#1{\big(#1\big)}%
\global\long\def\lr#1{\left(#1\right)}%

\global\long\def\ve{\varepsilon}%
\global\long\def\la{\lambda}%
\global\long\def\al{\alpha}%
\global\long\def\be{\beta}%
\global\long\def\ga{\gamma}%
\global\long\def\La{\Lambda}%
\global\long\def\De{\Delta}%
\global\long\def\na{\nabla}%

\global\long\def\ep{\epsilon}%
\global\long\def\fl{\flat}%
\global\long\def\sh{\sharp}%
\global\long\def\calN{\mathcal{N}}%
\global\long\def\supp{\mathrm{supp}}%

\title[Strichartz estimates and the cubic NLS on $\T^{2}$]{Strichartz estimates and global well-posedness of the cubic NLS on $\T^{2}$}
\author{Sebastian Herr}
\address{Fakultat f\"ur Mathematik, Universit\"at Bielefeld, Postfach 10 01 31, 33501 Bielefeld, Germany}
\email{herr@math.uni-bielefeld.de}
\author{Beomjong Kwak}
\address{Department of Mathematical Sciences, KAIST, 291 Daehak-ro, Yuseong-gu, Daejeon, Korea}
\email{beomjong@kaist.ac.kr}
\begin{abstract}
The optimal $L^4$-Strichartz estimate for the Schr{\"o}dinger equation on the two-dimensional rational torus
$\T^2$ is proved, which improves an estimate of Bourgain. A new method
based on incidence geometry is used. The approach yields a stronger $L^4$ bound on a
logarithmic time scale, which implies
global existence of solutions to
the cubic (mass-critical) nonlinear Schr\"odinger equation in $H^s(\T^2)$ for any $s>0$ and data which is small in the critical norm.
\end{abstract}

\maketitle
\section{Introduction}\label{sec:intro}
In the seminal work \cite{bourgain1993fourier} Bourgain proved Strichartz estimates for
the Schr\"odinger equation on (rational) tori $\T^{d}:=(\R/2\pi\Z)^{d}$. More precisely, in
dimension $d=2$, the endpoint estimate in \cite{bourgain1993fourier} states that there exists $c>0$ such that for all $\phi \in L^2(\T^2)$ and $N \in \N$,
\[ \|e^{it\Delta} P_N
\phi \|_{L^4_{t,x}([0,2\pi]\times \T^2)}\leq C_N  \|\phi\|_{L^2(\T^2)},
\text{ where } C_N= c\exp\Big(c\frac{\log(N)}{\log\log (N)}\Big).
\]
The proof in \cite{bourgain1993fourier} is based on the circle method and can be
reduced to an estimate for the number of divisors function, which
necessitates the above constant $C_N$.
However, in the example $\widehat{\phi}=\chi_{[-N,N]^2\cap \Z^2}$ we have
\begin{equation}\label{eq:ex-sharp}
    \|e^{it\Delta} P_N \phi \|_{L^4_{t,x}([0,2\pi]\times \T^2)}\approx (\log
N)^{1/4} \|\phi\|_{L^2(\T^2)},
\end{equation}
see \cite{bourgain1993fourier,takaoka20012d,kishimoto2014remark}.

More recently, the breakthrough result of Bourgain--Demeter on Fourier decoupling \cite{bourgain2015proof} provided a more robust approach which has significantly extended the range of available Strichartz estimates on rational and irrational tori. However, the above endpoint $L^4$ estimate has not been improved by this method.
Here, we will consider dimension $d=2$ only, but let us remark that in dimension $d=1$ there is a similar problem concerning the $L^6$ estimate, where it is known from \cite{bourgain1993fourier} that the best constant is between $c (\log N)^{1/6}$ and $C_N$, with  recent improvements of the upper bound to $c(\log N)^{2+\varepsilon}$ \cite{guth2020improved,guo-li-yung} by Fourier decoupling techniques.

In this paper, we obtain the sharp $L^4$ estimate in dimension $d=2$ by using methods of incidence geometry. 
Set $\log x:=\max\{1,\log_e x\}$ for $x>0$.
\begin{thm}
\label{thm:log Stri}There exists $c>0$, such that for all bounded sets $S\subset\Z^{2}$ and all $\phi\in L^{2}(\T^{2})$ we have 
\begin{equation}
\norm{e^{it\De}P_{S}\phi}_{L_{t,x}^{4}([0,2\pi]\times\T^{2}])}\leq c \left(\log\#S\right)^{1/4}\norm{\phi}_{L^{2}}.\label{eq:log Stri}
\end{equation}
\end{thm}

In fact, we prove a stronger result.
\begin{thm}
  \label{thm:local Stri}
There exists $c>0$, such that for all bounded sets $S\subset\Z^{2}$ and all $\phi\in L^{2}(\T^{2})$ we have 
\begin{equation}
\norm{e^{it\De}P_{S}\phi}_{L_{t,x}^{4}([0,\frac{1}{\log\#S}]\times\T^{2}])}\leq c \norm{\phi}_{L^{2}}.\label{eq:localStri}
\end{equation}
\end{thm}
\begin{rem}\label{rmk:stronger}
Theorem \ref{thm:local Stri} implies Theorem \ref{thm:log Stri}:
Applying
\eqref{eq:localStri} to each interval $[2\pi\frac{k-1}{m},2\pi\frac{k}{m}]$, $k=1,\ldots,m$,
for $m\approx\log\#S$, we obtain \eqref{eq:log Stri}.
In particular, \eqref{eq:ex-sharp} implies the sharpness of Theorem \ref{thm:local Stri} as well.
\end{rem}
For the proof of Theorem \ref{thm:local Stri} we develop a new method based on a counting argument for parallelograms with vertices in given sets, which relies on the Szemer\'{e}di-Trotter Theorem.
We remark that the Szemer\'{e}di-Trotter Theorem was previously used to bound the number of right triangles with vertices in a given set \cite{pach1992repeated} and it has also been introduced in \cite{bourgain2015proof} in connection to Fourier decoupling and discrete Fourier restriction theory. More precisely, if $\widehat{\phi}=\chi_S$, estimate \eqref{eq:log Stri} is a corollary of the Pach-Sharir bound in \cite{pach1992repeated}. We point out that in our proof of Theorem \ref{thm:local Stri} we also make use of the fourth vertex.

Theorems \ref{thm:log Stri} and \ref{thm:local Stri} apply to functions with Fourier support in arbitrary sets. While we make use of the lattice structure, we only use an elementary number theoretic argument in the proof of Theorem \ref{thm:local Stri}: In the parallelogram $(\xi_1,\xi_2,\xi_3,\xi_4)\in (\Z^2)^4$ the quantity $\tau=2(\xi_1-\xi_2)\cdot(\xi_1-\xi_4)$ must be a multiple of the greatest common divisor of the two coordinates of $\xi_1-\xi_4$, which is used to avoid a logarithmic loss in Theorem \ref{thm:local Stri}.

The $L^4$-Strichartz estimate plays a distinguished role in the analysis of the cubic nonlinear Schr\"{o}dinger equation (cubic NLS)
\begin{equation}
iu_{t}+\De u=\pm|u|^{2}u, \qquad u|_{t=0}=u_{0}\in H^{s}(\T^{2}),
\tag{{NLS}}\label{eq:NLS}
\end{equation}
which is $L^2(\T^2)$-critical. \eqref{eq:NLS}
is known to be locally well-posed in Sobolev spaces $H^s(\T^2)$ for $s>0$ due to \cite{bourgain1993fourier}.
It is also known \cite[Cor.\ 1.3]{kishimoto2014remark} that the Cauchy problem is not perturbatively well-posed in $L^2(\T^2)$, which is closely related to the example \eqref{eq:ex-sharp} discussed above.

By the conservation of energy, local well-posedness in $H^1(\T^2)$ implies global well-posedness for small enough data \cite[Theorem 2]{bourgain1993fourier}. In the
defocusing case, this has been refined to global well-posedness in
$H^s(\T^2)$ for $s>3/5$; see \cite{de2007global,fan2018bilinear,schippa2023improved}.
Additionally, the result in \cite{iteam-wt} shows that energy is transferred from small to higher frequencies and therefore causing growth of Sobolev norms $\|u(t)\|_{H^s}$ for $s>1$.

Theorem \ref{thm:local Stri} has the following consequence:
\begin{thm}
\label{thm:GWP}There exists $\delta>0$ such that for $s>0$ and initial data $u_{0}\in H^{s}(\T^2)$ with $\norm{u_{0}}_{L^{2}(\T^2)}\leq \delta$ the Cauchy problem \eqref{eq:NLS} is globally
well-posed.
\end{thm}

The proof is based on an estimate showing that $\|u(t)\|_{H^s(\T^2)}$ can grow only by a fixed multiplicative constant on a logarithmic time scale and because of
$\sum_{N\in 2^\N} 1/\log N=\infty$, any finite time interval can be covered. This argument crucially relies on the sharpness of
the estimate in Theorem \ref{thm:local Stri}.
Indeed, if the time interval in Theorem \ref{thm:local Stri} were $[0,(\log \#S)^{-\alpha}]$ for $\alpha>1$ instead, the sum would be $\sum_{N\in2^\N}1/(\log N)^\alpha<\infty$, which would not yield a global result.

\subsection*{Outline of the paper}\label{subsec:outline}
In Section \ref{sec:pre}, we introduce notations and recall the Szemer\'{e}di-Trotter Theorem. In Section \ref{sec:proof} we provide the proof of Theorem \ref{thm:local Stri}. Finally, in Section \ref{sec:proof-gwp} we prove the global well-posedness result, i.e.\ Theorem \ref{thm:GWP}.

\section{Preliminaries}\label{sec:pre}
We write $A\les B$ if $A\le CB$ for some universal constant $C>0$, and $A\approx B$
if both $A\les B$ and $B\les A$.
Given a set $E$, we denote $\chi_{E}$ as the sharp cutoff at $E$.

For proposition $P$, denote by $1_{P}$ the indicator function $$1_{P}:=\begin{cases}
1 & \text{, \ensuremath{P} is true}\\
0 & \text{, otherwise}
\end{cases}.$$

For a
function $f:\T^{2}\rightarrow\C$, $\mc{F}f=\widehat{f}$ denotes the Fourier series of $f$. For $S\subset\Z^{2}$,
we denote by $P_{S}$ the Fourier multiplier $\widehat{P_{S}f}:=\chi_{S}\cdot\widehat{f}$.
$2^{\N}$ denotes the set of dyadic numbers. For dyadic number $N\in2^{\N}$,
we denote by $P_{\le N}$ the sharp Littlewood-Paley cutoff $P_{\le N}f:=P_{[-N,N]^{2}}f$.
We denote $P_{N}:=P_{\le N}-P_{\le N/2}$, where we set $P_{\le1/2}:=0$.
For function $\phi:\T^{2}\rightarrow\C$ and time $t\in\R$, we define
$e^{it\De}\phi$ as the function such that
\[
\widehat{e^{it\De}\phi}(\xi)=e^{- it\left|\xi\right|^{2}}\widehat{\phi}(\xi).
\]

For simplicity, we denote $u_{N}=P_{N}u$ and $u_{\le N}=P_{\le N}u$,
for $u:\T^2\rightarrow\C$.

\subsection*{Geometric notations on $\Z^{2}$}
For integer point $(a,b)\in\Z^{2}$, ${(a,b)}^{\perp}$ denotes $(-b,a)$.

For integer point $(a,b)\in\Z^{2}\setminus\left\{ 0\right\} $, $\gcd\left((a,b)\right)$
denotes $\gcd(a,b)$.

Given two integer points $\xi_1,\xi_2\in \Z^2$, $\overleftrightarrow{\xi_1\xi_2}$ denotes the line through $\xi_1$ and $\xi_2$.

A \emph{parallelogram} is a quadruple $Q=(\xi_{1},\xi_{2},\xi_{3},\xi_{4})\in(\Z^{2})^{4}$
such that $\xi_{1}+\xi_{3}=\xi_{2}+\xi_{4}$. The set of all parallelograms is denoted by $\mc Q$.
\begin{figure}[!h]
\centering
\begin{tikzpicture}
\centering
  \coordinate [label=left:$\xi_1$]  (1) at (0,0);
  \coordinate [label=right:$\xi_4$] (2) at (2.5,0.5);
\coordinate [label=right:$\xi_3$] (3) at (3.5,-0.5);
\coordinate [label=left:$\xi_2$] (4) at (1,-1);
\coordinate [label=center:$Q$] (5) at (1.75,-0.25);
\draw (1) -- (2);
\draw (2) -- (3);
\draw (3) -- (4);
\draw (4) -- (1);
\foreach \x in {(1), (2), (3), (4)}
        \fill \x circle[radius=2pt];
\end{tikzpicture}
\caption{Parallelogram $Q$}
\end{figure}
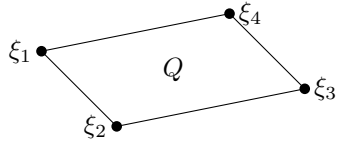
\emph{Segments} and \emph{points} are two-element pairs and elements of $\Z^{2}$, respectively.
We call by the edges of $Q$ either the segments $(\xi_{1},\xi_{2}),(\xi_{2},\xi_{3}),(\xi_{3},\xi_{4}),(\xi_{4},\xi_{1})$,
or the vectors $\pm\left(\xi_{1}-\xi_{2}\right),\pm\left(\xi_{1}-\xi_{4}\right)$.

For parallelogram $Q=(\xi_{1},\xi_{2},\xi_{3},\xi_{4})\in\mc Q$, we denote by $\tau_{Q}$ the number
\[
\tau_{Q}=\tau(\xi_{1},\xi_{2},\xi_{3},\xi_{4})=\left|\left|\xi_{1}\right|^{2}-\left|\xi_{2}\right|^{2}+\left|\xi_{3}\right|^{2}-\left|\xi_{4}\right|^{2}\right|=2\left|\left(\xi_{1}-\xi_{2}\right)\cdot\left(\xi_{1}-\xi_{4}\right)\right|.
\]
For $\tau\in\N$, we denote by $\mc Q^{\tau}$ the set of parallelograms $Q\in\mc Q$ such that $\tau_{Q}=\tau$. Thus, in particular, $\mc Q^{0}$ is the set of rectangles.

\subsection*{Szemer\'{e}di-Trotter}
The following is a consequence of Szemer\'{e}di-Trotter theorem of incidence
geometry.
\begin{prop}[{\cite[Corollary 8.5]{tao2006additive}}]
\label{prop:SzTr}Let $S\subset\R^{2}$ be a set of $n$ points,
where $n\in\N$. Let $k\ge2$ be an integer. The number $m$ of lines
in $\R^{2}$ passing through at least $k$ points of $S$ is bounded
by
\begin{equation}
m\les\frac{n^{2}}{k^{3}}+\frac{n}{k}.\label{eq:SzTr}
\end{equation}
\end{prop}

\begin{rem}\label{rmk:sztr-sharp}
An optimizer $S$ for \eqref{eq:SzTr} is a lattice $S=\Z^{2}\cap[-N,N]^{2},N\in\N$.
\end{rem}

\section{Proof of Theorem \ref{thm:local Stri}}\label{sec:proof} 

In this section, we prove Theorem \ref{thm:local Stri}. We will first
reduce Theorem \ref{thm:local Stri} to Proposition \ref{prop:inductive},
then to showing Lemma \ref{lem:3ineq}. Then we will finish the proof by
showing Lemma \ref{lem:3ineq}.

The proof of Theorem \ref{thm:local Stri} will be reduced to the following proposition.
\begin{prop} \label{prop:inductive}
Let $f:\Z^{2}\rightarrow[0,\infty)$ be a function
of the form
\[
f=\sum_{j=0}^{m}\la_{j}2^{-j/2}\chi_{S_{j}},
\]
where $S_{0},\ldots,S_{m},m\ge 1$ are disjoint subsets of $\Z^{2}$
such that $\#S_{j}\le2^{j}$, and $\la_{0},\ldots,\la_{m}\ge0$. Suppose
that for each $j=0,\ldots,m$ and $\xi\in S_{j}$, there exists at
most one line $\ell\ni\xi$ such that $\#(\ell\cap S_{j})\ge2^{j/2+C}$.
Then, we have
\begin{equation}
\sum_{Q\in\mc Q^{0}}f(Q)\les m\cdot\norm{\la_{j}}_{\ell_{j\le m}^{2}}^{4}\label{eq:sum_Q^0  f(Q) < m |la|_2^4}
\end{equation}
and
\begin{equation}
\sup_{M\in2^{\N}}\frac{1}{M}\sum_{\tau\approx M}\sum_{Q\in\mc Q^{\tau}}f(Q)\les\norm{\la_{j}}_{\ell_{j\le m}^{2}}^{4}.\label{eq:1/M sum_t~M sum_Q^t f(Q) < |la|_2^4}
\end{equation}
Here, $C>0$ is a uniform constant to be specified shortly, and
$f(Q)$ denotes $f(\xi_{1})f(\xi_{2})f(\xi_{3})f(\xi_{4})$ for parallelogram
$Q=(\xi_{1},\xi_{2},\xi_{3},\xi_{4})$.
\end{prop}

\begin{proof}[Proof of Theorem \ref{thm:local Stri} (assuming Proposition
\ref{prop:inductive})]
Let $S\subset\Z^{2}$ be a bounded set. Let
$m$ be the least integer greater than $\log_{2}\#S$. Since $\frac{1}{\log\#S}\les\frac{1}{m}$,
to prove Theorem \ref{thm:local Stri}, we only need to show for $\phi\in L^{2}(\T^{2})$
that
\begin{equation}
\norm{e^{it\De}P_{S}\phi}_{L_{t,x}^{4}([0,\frac{1}{m}]\times\T^{2})}\les\norm{\phi}_{L^{2}(\T^{2})}.\label{eq:P_S phi Stri}
\end{equation}
Decomposing $\widehat{\phi}=\sum_{k=0}^{3}i^{k}\widehat{\phi}_{k}$, $\widehat{\phi}_{k}\ge0$,
it suffices to show that for $f:\Z^{2}\rightarrow[0,\infty)$ supported
in $S$,
\begin{equation}
\norm{e^{it\De}\F^{-1}f}_{L_{t,x}^{4}([0,\frac{1}{m}]\times\T^{2})}\les\norm f_{\ell^{2}(\Z^{2})}.\label{eq:F^-1 f Stri}
\end{equation}

We define a sequence $\left\{ f_{n}\right\} $ of functions $f_{n}:\Z^{2}\rightarrow[0,\infty),\supp(f_{n})\subset S$
inductively. Let $f_{0}:=f$. Given $n\in\N$ and a function $f_{n}$,
we choose an enumeration $\xi_{1},\xi_{2},\ldots$ of $\Z^{2}$ (which
may depend on $n$) such that $f_{n}(\xi_{1})\ge f_{n}(\xi_{2})\ge\ldots$.
Let $S_{j}^{0}:=\left\{ \xi_{2^{j}},\ldots,\xi_{2^{j+1}-1}\right\} $
and $\la_{j}:=2^{j/2}f_n(\xi_{2^{j}})$ for $j=0,\ldots,m$. We have
\begin{equation}
\#S_{j}^{0}=2^{j}\label{eq:=000023S^0_j=00003D2^j}
\end{equation}
and
\begin{align}
\norm{\la_{j}}_{\ell_{j\le m}^{2}}&=\norm{2^{j/2}f_{n}(\xi_{2^{j}})}_{\ell_{j\le m}^{2}}\label{eq:|la|_2<|f_n|_2}\\&
\les f_{n}(\xi_{1})+\norm{\sum_{j=1}^{m}f_{n}(\xi_{2^{j}})\chi_{\left\{ \xi_{2^{j-1}+1},\ldots,\xi_{2^{j}}\right\} }}_{\ell^{2}(\Z^{2})}\nonumber \\&
\les\norm{f_{n}}_{\ell^{2}(\Z^{2})}.\nonumber 
\end{align}
For $j=0,\ldots,m$, we define $E_{j}\subset S_{j}^{0}$ as the set
of intersections $\xi\in S_{j}^{0}$ of two lines $\ell_{1},\ell_{2}$
such that
\[
\#\left(\ell_{1}\cap S_{j}^{0}\right),\#\left(\ell_{2}\cap S_{j}^{0}\right)\ge2^{j/2+C}.
\]
By the Szemer\'{e}di-Trotter bound \eqref{eq:SzTr} and \eqref{eq:=000023S^0_j=00003D2^j}, we have
\begin{align}
\sqrt{\#E_{j}} & \le\#\left\{ \ell\subset\R^{2}:\ell\text{ is a line and }\#\left(\ell\cap S_{j}^{0}\right)\ge2^{j/2+C}\right\} \label{eq:E_j bound}\\
 & \les(\#S_{j}^{0})^{2}/(2^{j/2+C})^3+\#S_{j}^{0}/2^{j/2+C}\nonumber \\
 & \les2^{j/2-C}.\nonumber 
\end{align}

Let $f_{n+1}:\Z^{2}\rightarrow[0,\infty)$ be the function
\[
f_{n+1}:=f_{n}\chi_{E}, \quad E:=\bigcup_{j=0}^{m}E_{j}.
\]
Since $f_{n}(\xi)\le f_{n}(\xi_{2^{j}})=\la_{j}2^{-j/2}$ holds for
$\xi\in E_{j}\subset S_{j}^{0}$, by \eqref{eq:E_j bound} and \eqref{eq:|la|_2<|f_n|_2},
we have
\[
\norm{f_{n+1}}_{\ell^{2}(\Z^{2})}=\norm{f_{n}\chi_{E}}_{\ell^{2}(\Z^{2})}\les\norm{\la_{j}2^{-j/2}\cdot\sqrt{\#E_{j}}}_{\ell_{j\le m}^{2}}\les2^{-C}\norm{f_{n}}_{\ell^{2}(\Z^{2})}.
\]
Fixing $C\in\N$ as a big number gives
\[
\norm{f_{n+1}}_{\ell^{2}(\Z^{2})}\le\frac{1}{2}\norm{f_{n}}_{\ell^{2}(\Z^{2})},
\]
which implies
\begin{equation}
\norm{f_{n}}_{\ell^{2}(\Z^{2})}\le\frac{1}{2}\norm{f_{n-1}}_{\ell^{2}(\Z^{2})}\le\ldots\le2^{-n}\norm f_{\ell^{2}(\Z^{2})}.\label{eq:f_n _2 < 2^-n f_2}
\end{equation}

Let $S_{j}:=S_{j}^{0}\setminus E_{j}$. By the definition of $E_{j}$,
the function
\[
g_{n}:=\sum_{j=0}^{m}\la_{j}2^{-j/2}\chi_{S_{j}}
\]
satisfies the conditions for Proposition \ref{prop:inductive}.
Since $f_n(\xi)\le f_n(\xi_{2^{j}})=\la_{j}2^{-j/2}$ holds for $\xi\in S_{j}\subset S_{j}^{0}=\left\{ \xi_{2^{j}},\ldots,\xi_{2^{j+1}-1}\right\} $,
we have
\begin{equation}
h_{n}:=f_{n}-f_{n+1}=\sum_{j=0}^{m}f_{n}\chi_{S_{j}}\le\sum_{j=0}^{m}\la_{j}2^{-j/2}\chi_{S_{j}}=g_{n}.\label{eq:h_n < g_n}
\end{equation}

Denoting $T_{0}:=\frac{1}{m}$, by \eqref{eq:h_n < g_n} we have
\begin{align*}
\int_{0}^{T_{0}}\int_{\T^{2}}\left|e^{it\De}\F^{-1}h_{n}\right|^{4}dxdt & \le\frac{1}{T_{0}}\int_{0}^{2T_{0}}\int_{0}^{T}\int_{\T^{2}}\left|e^{it\De}\F^{-1}h_{n}\right|^{4}dxdtdT\\
 & \approx\frac{1}{T_{0}}\int_{0}^{2T_{0}}\int_{0}^{T}\F\left(\left|e^{it\De}\F^{-1}h_{n}\right|^{4}\right)(0)dtdT\\
 & \approx\frac{1}{T_{0}}\sum_{Q\in\mc Q}h_{n}\left(Q\right)\cdot\Re\int_{0}^{2T_{0}}\int_{0}^{T}e^{it\tau_{Q}}dtdT\\
 & \approx\sum_{Q\in\mc Q}h_{n}\left(Q\right)\cdot\frac{1-\cos\left(2T_{0}\tau_{Q}\right)}{T_{0}\tau_{Q}^{2}}\\
 & \les T_{0}\sum_{Q\in\mc Q^{0}}h_{n}(Q)+\sum_{\substack{\tau>0}
}\min\left\{ T_{0},\frac{1}{T_{0}\tau^{2}}\right\} \sum_{Q\in\mc Q^{\tau}}h_{n}(Q)\\
 & \les T_{0}\sum_{Q\in\mc Q^{0}}g_{n}(Q)+\sum_{\substack{\tau>0}
}\min\left\{ T_{0},\frac{1}{T_{0}\tau^{2}}\right\} \sum_{Q\in\mc Q^{\tau}}g_{n}(Q)
\end{align*}
and
\begin{align*}
 & \sum_{\substack{\tau>0}
}\min\left\{ T_{0},\frac{1}{T_{0}\tau^{2}}\right\} \sum_{Q\in\mc Q^{\tau}}g_{n}(Q)\\
 \les{}&\sum_{M\in2^{\N}}\min\left\{ T_{0}M,\frac{1}{T_{0}M}\right\} \frac{1}{M}\sum_{\tau\approx M}\sum_{Q\in\mc Q^{\tau}}g_{n}(Q)\\
  \les{}&\sup_{M\in2^{\N}}\frac{1}{M}\sum_{\tau\approx M}\sum_{Q\in\mc Q^{\tau}}g_{n}(Q),
\end{align*}
concluding by \eqref{eq:sum_Q^0  f(Q) < m |la|_2^4}, \eqref{eq:1/M sum_t~M sum_Q^t f(Q) < |la|_2^4},
and \eqref{eq:|la|_2<|f_n|_2} that
\begin{align}\label{eq:h_n stri}
\norm{e^{it\De}\F^{-1}h_{n}}_{L^{4}([0,T_{0}]\times\T^{2})}&=\left(\int_{0}^{T_{0}}\int_{\T^{2}}\left|e^{it\De}\F^{-1}h_{n}\right|^{4}dxdt\right)^{1/4}\\
&\les\norm{\la_{j}}_{\ell_{j\le m}^{2}}\les\norm{f_{n}}_{\ell^{2}(\Z^{2})}.\nonumber
\end{align}
Writing $f=\sum_{n=0}^{\infty}(f_{n}-f_{n+1})=\sum_{n=0}^{\infty}h_{n}$,
by \eqref{eq:h_n stri} and \eqref{eq:f_n _2 < 2^-n f_2}, we have
\begin{align*}
\norm{e^{it\De}\F^{-1}f}_{L_{t,x}^{4}\left([0,T_{0}]\times\T^{2}\right)} & \le\sum_{n=0}^{\infty}\norm{e^{it\De}\F^{-1}h_{n}}_{L_{t,x}^{4}\left([0,T_{0}]\times\T^{2}\right)}\\
 & \les\sum_{n=0}^{\infty}\norm{f_{n}}_{\ell^{2}(\Z^{2})}\\
 & \les\sum_{n=0}^{\infty}2^{-n}\norm f_{\ell^{2}(\Z^{2})}\les\norm f_{\ell^{2}(\Z^{2})},
\end{align*}
which is \eqref{eq:F^-1 f Stri} and therefore completes the proof of Theorem
\ref{thm:local Stri}.
\end{proof}

A \emph{cross} is a triple $(\xi,\ell_{1},\ell_{2})$ of two mutually
orthogonal lines $\ell_{1},\ell_{2}$ and their intersection $\xi$.
For $\{S_j\}_{j=0}^m$ as in Proposition \ref{prop:inductive}, we categorize crosses $\left(\xi,\ell_{1},\ell_{2}\right),\xi\in\cup_{j=0}^m S_j$ into three
types
\[
\begin{cases}
\text{Type 1} & \text{ if } a\ge j/2+C\\
\text{Type 2} & \text{ if } 1\le a<j/2+C\\
\text{Type 3} & \text{ if }a=0,
\end{cases}
\]
where $j$ is the index such that $\xi\in S_j$, and $a$ is the number
\[
a=\log_{2}\max\left\{ \#\left(\ell_{1}\cap S_{j}\right),\#\left(\ell_{2}\cap S_{j}\right)\right\}.
\]
Note that $a\in\{0\}\cup[1,\infty)$ since $\ell_1\cap S_j$ is nonempty.

Given a rectangle $\left(\xi_{1},\xi_{2},\xi_{3},\xi_{4}\right)$
of four distinct vertices, its vertex $\xi_{1}$ is called a \emph{vertex
of type }$\al$, $\al=1,2,3$, if the cross $(\xi_{1},\overleftrightarrow{\xi_{1}\xi_{2}},\overleftrightarrow{\xi_{1}\xi_{4}})$
is of type $\al$.

For $\al,\be=1,2,3$, we denote by $\mc Q_{\al,\be}^{0}$ the set
of rectangles $\left(\xi_{1},\xi_{2},\xi_{3},\xi_{4}\right)\in\mc Q^{0}$
of four distinct vertices $\xi_1,\xi_2,\xi_3,\xi_4\in\cup_{j=0}^m S_j$ such that $\xi_{1},\xi_{2}$ are type $\al$-vertices
and $\xi_{3},\xi_{4}$ are type $\be$-vertices. Although the union
of $\mc Q_{\al,\be}^{0}$ is only a proper subcollection of $\mc Q^{0}$,
the following lemma provides a reduction to counting rectangles in
$\mc Q_{\al,\be}^{0}$.
\begin{lem}
Let $f$ and $\left\{ S_{j}\right\}_{j=0}^m$ be as in Proposition \ref{prop:inductive}.
Let $\tau\ge0$ be an integer. We have
\begin{equation}
\sum_{Q\in\mc Q^{\tau}}f(Q)\les\max_{\al,\be=1,2,3}\sum_{\substack{Q=(\xi_{1},\xi_{2},\xi_{3},\xi_{4})\in\mc Q_{\al,\be}^{0}\\
\gcd(\xi_{1}-\xi_{4})\mid\tau
}
}f(Q)+\norm f_{\ell^{2}(\Z^{2})}^{4}.\label{eq:sum_Q^t f(Q) < max_ab sum Q_ab f(Q)+f_2^4}
\end{equation}
\end{lem}

\begin{proof}
For $\xi\in\Z^{2}\setminus\left\{ 0\right\} $ and $\s\in\Z$,
we denote by $\mc E_{\xi}^{\s}$ the set of segments $(\xi_{1},\xi_{4})\in(\Z^{2})^{2}$
such that $\xi_{1}-\xi_{4}=\xi$ and $\xi_{1}\cdot\xi=\s$.

Since $\tau_{Q}=2\left|(\xi_{1}-\xi_{2})\cdot(\xi_{1}-\xi_{4})\right|$
is a multiple of $\gcd(\xi_{1}-\xi_{4})$ for any parallelogram $Q=(\xi_{1},\xi_{2},\xi_{3},\xi_{4})$
such that $\xi_{1}-\xi_{4}\neq0$, we have
\begin{align*}
\sum_{Q\in\mc Q^{\tau}}f(Q) & \les\sum_{\xi\in\Z^{2}\setminus\left\{ 0\right\} }\sum_{\substack{Q=(\xi_{1},\xi_{2},\xi_{3},\xi_{4})\in\mc Q^{\tau}\\
\xi_{1}-\xi_{4}=\xi
}
}f(Q)+\sum_{\xi_{1},\xi_{2}\in\Z^{2}}f(\xi_{1})^{2}f(\xi_{2})^{2}\\
 & \les\sum_{\substack{\xi\in\Z^{2}\setminus\left\{ 0\right\} \\
\gcd(\xi)\mid\tau
}
}\sum_{\substack{Q=(\xi_{1},\xi_{2},\xi_{3},\xi_{4})\in\mc Q^{\tau}\\
\xi_{1}-\xi_{4}=\xi
}
}f(Q)+\norm f_{\ell^{2}(\Z^{2})}^{4}
\end{align*}
and by Cauchy-Schwarz inequality,
{\allowdisplaybreaks
\begin{align*}
 & \sum_{\substack{\xi\in\Z^{2}\setminus\left\{ 0\right\} \\
\gcd(\xi)\mid\tau
}
}\sum_{\substack{Q=(\xi_{1},\xi_{2},\xi_{3},\xi_{4})\in\mc Q^{\tau}\\
\xi_{1}-\xi_{4}=\xi
}
}f(Q)\\
 & \les\sum_{\substack{\xi\in\Z^{2}\setminus\left\{ 0\right\} \\
\gcd(\xi)\mid\tau
}
}\sum_{\substack{\s_{1},\s_{2}\in\Z\\
\s_{1}-\s_{2}=\pm\tau/2
}
}\sum_{\substack{(\xi_{1},\xi_{4})\in\mc E_{\xi}^{\s_{1}}\\
(\xi_{2},\xi_{3})\in\mc E_{\xi}^{\s_{2}}
}
}f(\xi_{1})f(\xi_{4})f(\xi_{2})f(\xi_{3})\\
 & \les\sum_{\substack{\xi\in\Z^{2}\setminus\left\{ 0\right\} \\
\gcd(\xi)\mid\tau
}
}\sum_{\s\in\Z}\left(\sum_{(\xi_{1},\xi_{4})\in\mc E_{\xi}^{\s}}f(\xi_{1})f(\xi_{4})\right)^{2}\\
 & \les\max_{\al,\be=1,2,3}\sum_{\substack{\xi\in\Z^{2}\setminus\left\{ 0\right\} \\
\gcd(\xi)\mid\tau
}
}\sum_{\s\in\Z}\left(\sum_{\substack{(\xi_{1},\xi_{4})\in\mc E_{\xi}^{\s}\\
(\xi_1,\xi_1+\xi \R,\xi_1+\xi^\perp \R)\text{ is a cross of type }\al\\
(\xi_4,\xi_4+\xi \R,\xi_4+\xi^\perp \R)\text{ is a cross of type }\be
}
}f(\xi_{1})f(\xi_{4})\right)^{2}\\
 & \les\max_{\al,\be=1,2,3}\sum_{\substack{\xi\in\Z^{2}\setminus\left\{ 0\right\} \\
\gcd(\xi)\mid\tau
}
}\sum_{\s\in\Z}\sum_{\substack{
(\xi_1,\xi_4)\in\mc E^\s_\xi\\(\xi_{1},\xi_{2},\xi_{3},\xi_{4})\in\mc Q_{\al,\be}^{0}\text{ or }(\xi_2,\xi_3)=(\xi_1,\xi_4)
}
}f(\xi_1)f(\xi_4)f(\xi_2)f(\xi_3)\\
 & \les\max_{\al,\be=1,2,3}\sum_{\substack{\xi\in\Z^{2}\setminus\left\{ 0\right\} \\
\gcd(\xi)\mid\tau
}
}\left(\sum_{\substack{Q=(\xi_{1},\xi_{2},\xi_{3},\xi_{4})\in\mc Q_{\al,\be}^{0}\\
\xi_{1}-\xi_{4}=\xi
}
}f(Q)+\sum_{\xi_{1}-\xi_{4}=\xi}f(\xi_{1})^{2}f(\xi_{4})^{2}\right)\\
 & \les\max_{\al,\be=1,2,3}\sum_{\substack{Q=(\xi_{1},\xi_{2},\xi_{3},\xi_{4})\in\mc Q_{\al,\be}^{0}\\
\gcd(\xi_{1}-\xi_{4})\mid\tau
}
}f(Q)+\norm f_{\ell^{2}(\Z^{2})}^{4},
\end{align*}
}finishing the proof.
\end{proof}

There are three main inequalities to be shown. 
\begin{lem}\label{lem:3ineq}
Let $f$ and $\{\la_j\}_{j=0}^m$ be as in Proposition \ref{prop:inductive}.

In the cases
$(\al,\be)\neq(2,2)$ we have
\begin{equation}
\sum_{Q\in\mc Q_{\al,\be}^{0}}f(Q)\les\norm{\la_{j}}_{\ell_{j\le m}^{2}}^{4}.\label{eq:sum Q_ab f(Q) < la_2 ^4}
\end{equation}

In case $(\al,\be)=(2,2)$ we have
\begin{equation}
\sum_{Q\in\mc Q_{2,2}^{0}}f(Q)\les m\norm{\la_{j}}_{\ell_{j\le m}^{2}}^{4}\label{eq:sum Q_22 f(Q) < m*la_2^4}
\end{equation}
and
\begin{equation}
\sum_{Q=(\xi_{1},\xi_{2},\xi_{3},\xi_{4})\in\mc Q_{2,2}^{0}}\frac{1}{\gcd(\xi_{1}-\xi_{4})}f(Q)\les\norm{\la_{j}}_{\ell_{j\le m}^{2}}^{4}.\label{eq:sum Q_22 1/gcd f(Q) < la_2^4}
\end{equation}
\end{lem}

\begin{proof}[Proof of Proposition \ref{prop:inductive} assuming Lemma \ref{lem:3ineq}]

We first prove \eqref{eq:sum_Q^0  f(Q) < m |la|_2^4}, which concerns the case $\tau=0$. By \eqref{eq:sum_Q^t f(Q) < max_ab sum Q_ab f(Q)+f_2^4}, \eqref{eq:sum Q_ab f(Q) < la_2 ^4},
and \eqref{eq:sum Q_22 f(Q) < m*la_2^4}, we have
\[
\sum_{Q\in\mc Q^{0}}f(Q)\les\max_{\al,\be=1,2,3}\sum_{Q=(\xi_{1},\xi_{2},\xi_{3},\xi_{4})\in\mc Q_{\al,\be}^{0}}f(Q)+\norm f_{\ell^{2}(\Z^{2})}^{4}\les m\norm{\la_{j}}_{\ell_{j\le m}^{2}}^{4},
\]
which is just \eqref{eq:sum_Q^0  f(Q) < m |la|_2^4}.

Now we prove \eqref{eq:1/M sum_t~M sum_Q^t f(Q) < |la|_2^4}, which is for $\tau\neq 0$. By \eqref{eq:sum_Q^t f(Q) < max_ab sum Q_ab f(Q)+f_2^4}, for $M\in2^{\N}$,
we have
\[
\frac{1}{M}\sum_{\tau\approx M}\sum_{Q\in\mc Q^{\tau}}f(Q)\les\frac{1}{M}\max_{\al,\be=1,2,3}\sum_{\tau\approx M}\sum_{\substack{Q=(\xi_{1},\xi_{2},\xi_{3},\xi_{4})\in\mc Q_{\al,\be}^{0}\\
\gcd(\xi_{1}-\xi_{4})\mid\tau
}
}f(Q)+\norm f_{\ell^{2}(\Z^{2})}^{4},
\]
and for $\al,\be=1,2,3$, we have
\begin{align*}
 & \frac{1}{M}\sum_{\tau\approx M}\sum_{\substack{Q=(\xi_{1},\xi_{2},\xi_{3},\xi_{4})\in\mc Q_{\al,\be}^{0}\\
\gcd(\xi_{1}-\xi_{4})\mid\tau
}
}f(Q)\\
 & =\frac{1}{M}\sum_{\tau\approx M}\sum_{Q=(\xi_{1},\xi_{2},\xi_{3},\xi_{4})\in\mc Q_{\al,\be}^{0}}1_{\gcd(\xi_{1}-\xi_{4})\mid\tau}\cdot f(Q)\\
 & =\sum_{Q=(\xi_{1},\xi_{2},\xi_{3},\xi_{4})\in\mc Q_{\al,\be}^{0}}\left(\frac{1}{M}\sum_{\tau\approx M}1_{\gcd(\xi_{1}-\xi_{4})\mid\tau}\right)\cdot f(Q)\\
 & \les\sum_{Q=(\xi_{1},\xi_{2},\xi_{3},\xi_{4})\in\mc Q_{\al,\be}^{0}}\frac{1}{\gcd(\xi_{1}-\xi_{4})}\cdot f(Q),
\end{align*}
which is $O(\norm{\la_{j}}_{\ell^{2}_{j\le m}}^{4})$ by \eqref{eq:sum Q_ab f(Q) < la_2 ^4}
and \eqref{eq:sum Q_22 1/gcd f(Q) < la_2^4}, and finishes the proof
of \eqref{eq:1/M sum_t~M sum_Q^t f(Q) < |la|_2^4}.
\end{proof}

Before turning to the proof of Lemma \ref{lem:3ineq}, we consider two preparatory lemmas, where we use the following notation:

For vectors $\overrightarrow{j}=(j_1,j_2,j_3,j_4)\in \N^4$ and $\overrightarrow{a}=(a_1,a_2,a_3,a_4)\in\N^4$, we denote by $\mc Q^{0}(\overrightarrow{j},\overrightarrow{a})$ the set of rectangles $(\xi_{1},\xi_{2},\xi_{3},\xi_{4})\in\mc Q^{0}\cap\left(S_{j_{1}}\times S_{j_{2}}\times S_{j_{3}}\times S_{j_{4}}\right)$
of four distinct vertices such that
\begin{equation}
2^{a_{k}}\le\max\left\{ \#\left(\overleftrightarrow{\xi_{k}\xi_{k+1}}\cap S_{j_{k}}\right),\#\left(\overleftrightarrow{\xi_{k}\xi_{k-1}}\cap S_{j_{k}}\right)\right\} <2^{a_{k}+1},\label{eq:a_k}
\end{equation}
where the cyclic convention on index $\xi_{4l+k}=\xi_{k},l\in\Z$
is used.

\begin{figure}[!h]
\centering
\begin{tikzpicture}
\centering
\coordinate [label=above left:$\xi_1$]  (1) at (0,2.5);
\coordinate [label=below left:$\xi_2$] (2) at (0,0);
\coordinate [label=below right:$\xi_3$] (3) at (4,0);
\coordinate [label=above right:$\xi_4$] (4) at (4,2.5);

\coordinate [label=above left:$S_{j_1}$]  (S1) at (-0.75,2.5);
\coordinate [label=below left:$S_{j_2}$] (S2) at (-0.75,0);
\coordinate [label=below right:$S_{j_3}$] (S3) at (4.75,0);
\coordinate [label=above right:$S_{j_4}$] (S4) at (4.75,2.5);
\coordinate (1-) at (-2,2.5);
\coordinate (4+) at (6,2.5);
\coordinate (2-) at (-2,0);
\coordinate (1++) at (0,2.5+2);
\coordinate (2--) at (0,0-2);
\coordinate (4++) at (4,2.5+2);
\coordinate (3--) at (4,0-2);
\coordinate (3+) at (6,0);
\draw (1) -- (2);
\draw (2-) -- (3+);
\draw (1-) -- (4+);
\draw (1++) -- (2--);
\draw (4++) -- (3--);
\foreach \x in {(1), (2), (3), (4)}
{
    \fill \x circle[radius=2pt];
}
\foreach \x in {(1), (2), (3), (4)}
{
    \draw \x circle[radius=20pt];
}
\end{tikzpicture}
\caption{Rectangle $(\xi_1,\xi_2,\xi_3,\xi_4)\in \mc Q^{0}(\overrightarrow{j},\overrightarrow{a})$}
\end{figure}

\begin{lem}\label{lem:j1+j2+a3}
Let $\left\{ S_{j}\right\} _{j=0}^{m},m\ge 1$ be as in Proposition
\ref{prop:inductive}. Let $j_{1},j_{2},j_{3},j_{4},a_{3}\ge0$
be integers. Then, the number of rectangles $(\xi_1,\xi_2,\xi_3,\xi_4)\in\mc Q^0\cap (S_{j_1}\times S_{j_2}\times S_{j_3}\times S_{j_4})$ of four distinct vertices such that
\[
\#\left(\overleftrightarrow{\xi_2 \xi_3}\cap S_{j_3}\right)<2^{a_3+1}
\]
is $O(2^{j_{1}+j_{2}+a_{3}})$.
\end{lem}

\begin{proof}
There are at most $\#S_{j_1}\cdot\#S_{j_2}=O(2^{j_{1}+j_{2}})$ possible choices of $(\xi_{1},\xi_{2})\in S_{j_{1}}\times S_{j_{2}}$.
Once the pair of two vertices $(\xi_{1,}\xi_{2})\in S_{j_{1}}\times S_{j_{2}}$
is fixed, the third vertex $\xi_{3}$ should lie on the line $\ell_{23}\ni\xi_{2}$
orthogonal to $\overleftrightarrow{\xi_{1}\xi_{2}}$ and we require
\[
\#\left(\ell_{23}\cap S_{j_{3}}\right)=\#\left(\overleftrightarrow{\xi_{2}\xi_{3}}\cap S_{j_{3}}\right)<2^{a_{3}+1},
\]
so there are only $O(2^{a_{3}})$ possible choices of $\xi_{3}\in\ell_{23}$,
which then uniquely determines a rectangle. Therefore, we have $O(2^{j_{1}+j_{2}}\cdot2^{a_{3}})=O(2^{j_{1}+j_{2}+a_{3}})$ such rectangles.

\begin{figure}[!h]
\centering
\begin{tikzpicture}
\centering
\coordinate [label=above left:$\xi_1$]  (1) at (0,2.5);
\coordinate [label=below left:$\xi_2$] (2) at (0,0);
\coordinate [label=below right:$\xi_3$] (3) at (4,0);

\coordinate [label=above left:$S_{j_1}$]  (S1) at (-0.75,2.5);
\coordinate [label=below left:$S_{j_2}$] (S2) at (-0.75,0);
\coordinate [label=below right:$S_{j_3}$] (S3) at (4.75,0);
\coordinate [label=left:$\ell_{23}$](2-) at (-2,0);
\coordinate (3+) at (6,0);
\draw (1) -- (2);
\draw (2-) -- (3+);
\foreach \x in {(1), (2), (3)}
{
    \fill \x circle[radius=2pt];
    \draw \x circle[radius=20pt];
}
\end{tikzpicture}
\caption{Choice of $\xi_1,\xi_2,\xi_3$ in the proof of Lemma \ref{lem:j1+j2+a3}}
\end{figure}

\end{proof}
The following lemma is useful in the case that $\xi_{1}$ is a vertex of type $2$.
\begin{lem}\label{lem:counting}
Let $\left\{ S_{j}\right\} _{j=0}^{m},m\ge 1$ be as in Proposition
\ref{prop:inductive}. Let $j_{1},j_{2},j_{3},j_{4}$, $a_{1},a_{2},a_{3},a_{4}\ge0$
be integers. Assume that
\begin{equation}
1\le a_{1}<j_{1}/2+C.\label{eq:a1 bound}
\end{equation}
We have
\begin{equation}
\#\mc Q^{0}(\overrightarrow{j},\overrightarrow{a}) \les 2^{2j_{1}-2a_{1}+a_{2}+a_{4}},\label{eq:-2a1+a2+a4}
\end{equation}
\begin{equation}
\#\mc Q^{0}(\overrightarrow{j},\overrightarrow{a}) \les 2^{2j_{1}-2a_{1}+a_{2}+a_{3}},\label{eq:-2a1+a2+a3}
\end{equation}
and
\begin{equation}
\sum_{(\xi_{1},\xi_{2},\xi_{3},\xi_{4})\in\mc Q^{0}(\overrightarrow{j},\overrightarrow{a})  }\frac{1}{\gcd(\xi_{1}-\xi_{4})}\les 2^{2j_{1}-2a_{1}+a_{2}+a_{4}/2}.\label{eq:-2a1+a2+a4/2}
\end{equation}
\end{lem}

We note that the assumption \eqref{eq:a1 bound} is a priori necessary if $\xi_{1}$ is a vertex of type $2$.
\begin{proof}
By \eqref{eq:SzTr}, the number of lines $\ell$ such that $2^{a_{1}}\le\#(\ell\cap S_{j_{1}})<2^{a_1+1}$
is $O(2^{2j_{1}}\cdot2^{-3a_{1}}+2^{j_{1}}\cdot2^{-a_{1}})=O(2^{2j_{1}-3a_{1}})$, and for each such $\ell$, we have $O(2^{a_1})$ number of points $\xi_1\in\ell\cap S_{j_1}$.
Thus, there exist at most $O(2^{2j_{1}-2a_{1}})$ crosses $(\xi_{1},\ell_{12},\ell_{14})$
such that
\[
2^{a_{1}}\le\max\left\{ \#(\ell_{12}\cap S_{j_{1}}),\#(\ell_{14}\cap S_{j_{1}})\right\} <2^{a_{1}+1}.
\]
For such a cross $(\xi_{1},\ell_{12},\ell_{14})$ to be a corner of
a rectangle in $\mc Q^{0}(\overrightarrow{j},\overrightarrow{a})  $, for \eqref{eq:a_k}
we require further that
\begin{equation}
\#(\ell_{12}\cap S_{j_{2}})<2^{a_{2}+1}\label{eq:12}
\end{equation}
and
\begin{equation}
\#(\ell_{14}\cap S_{j_{4}})<2^{a_{4}+1}.\label{eq:14}
\end{equation}

By \eqref{eq:12}, there exist at most $O(2^{a_{2}})$ choices of
vertices $\xi_{2}\in\ell_{12}\cap S_{j_{2}}$.

\begin{figure}[!h]
\centering
\begin{tikzpicture}
\centering
\coordinate [label=above left:$\xi_1$]  (1) at (0,2.5);
\coordinate [label=below left:$\xi_2$] (2) at (0,0);
\coordinate  (3) at (4,0);
\coordinate  (4) at (4,2.5);

\coordinate [label=above left:$S_{j_1}$]  (S1) at (-0.75,2.5);
\coordinate [label=below left:$S_{j_2}$] (S2) at (-0.75,0);
\coordinate [label=below right:$S_{j_3}$] (S3) at (4.75,0);
\coordinate [label=above right:$S_{j_4}$] (S4) at (4.75,2.5);
\coordinate [label=left:$\ell_{14}$](1-) at (-2,2.5);
\coordinate (4+) at (6,2.5);
\coordinate [label=left:$\ell_{23}$](2-) at (-2,0);
\coordinate [label=above:$\ell_{12}$](1++) at (0,2.5+2);
\coordinate (2--) at (0,0-2);
\coordinate (3+) at (6,0);
\draw (1) -- (2);
\draw (2-) -- (3+);
\draw (1-) -- (4+);
\draw (1++) -- (2--);
\foreach \x in {(1), (2)}
{
    \fill \x circle[radius=2pt];
}
\foreach \x in {(1), (2), (3), (4)}
{
    \draw \x circle[radius=20pt];
}
\end{tikzpicture}
\caption{Choice of $\xi_1$ and $\xi_2$ in the proof of Lemma \ref{lem:counting}}
\end{figure}

Having fixed $\xi_{1}$
and $\xi_{2}$, we choose either $\xi_{3}$ or $\xi_{4}$ as follows,
which then uniquely determines a rectangle $(\xi_{1},\xi_{2},\xi_{3},\xi_{4})\in\mc Q^{0}$.
\begin{itemize}
\item \emph{Choice of $\xi_{4}$. }Since the choice of $\xi_{4}\in\ell_{14}\cap S_{j_{4}}$
in advance uniquely determines a rectangle, by \eqref{eq:14}, we have
\eqref{eq:-2a1+a2+a4}. Also, labeling $\ell_{14}\cap S_{j_{4}}\setminus\left\{ \xi_{1}\right\} =:\left\{ \xi_{4}^{1},\ldots,\xi_{4}^{l}\right\} ,l<2^{a_{4}+1}$,
we have
\[
\sum_{r=1}^{l}\frac{1}{\gcd(\xi_{1}-\xi_{4}^{r})}\les\frac{1}{1}+\cdots+\frac{1}{l}\les\log l\les2^{a_{4}/2},
\]
which implies \eqref{eq:-2a1+a2+a4/2}.
\item \emph{Choice of $\xi_{3}$.} We can also determine a rectangle by
choosing $\xi_{3}\in\ell_{23}\cap S_{j_{3}}$, where $\ell_{23}\ni\xi_{2}$
is the line parallel with $\ell_{14}$. To form a rectangle in $\mc Q^{0}(\overrightarrow{j},\overrightarrow{a})  $,
we require
\[
\#(\ell_{23}\cap S_{j_{3}})=\#(\overleftrightarrow{\xi_{2}\xi_{3}}\cap S_{j_{3}})<2^{a_{3}+1},
\]
so there are at most $O(2^{a_{3}})$ choices of such vertices $\xi_{3}$.
Thus, we have \eqref{eq:-2a1+a2+a3}.\qedhere
\end{itemize}
\end{proof}

We can now lay the last brick of the proof of Proposition \ref{prop:inductive}.
\begin{proof}[Proof of Lemma \ref{lem:3ineq}]
We split the proof  into the cases
(i) $\al=1$ (or $\be=1$), (ii) $(\al,\be)=(2,2)$, (iii) $(\al,\be)=(3,3)$, and (iv) $(\al,\be)=(2,3)$ (or $(3,2)$).\\

\emph{Case I: $\al=1$ (or $\be=1$).}

For $\xi_{1}\in S_{j_{1}},j_{1}\in\N$, by the assumption of Proposition
\ref{prop:inductive} there exists at most one line $\ell_{\xi_{1}}\ni\xi_1$
such that $\#(\ell_{\xi_{1}}\cap S_{j_{1}})\ge2^{j_1/2+C}$. Thus, for any
rectangle $Q=(\xi_{1},\xi_{2},\xi_{3},\xi_{4})\in\mc Q_{1,\be}^{0}$,
to which the inequality 
\[\max\left\{ \#\left(\overleftrightarrow{\xi_{1}\xi_{2}}\cap S_{j_{1}}\right),\#\left(\overleftrightarrow{\xi_{1}\xi_{4}}\cap S_{j_{1}}\right)\right\} \ge2^{j_1/2+C}\]
applies since $\xi_{1}$ is of type $\al=1$, we have either
$\xi_{2}\in\ell_{\xi_{1}}$ or $\xi_{4}\in\ell_{\xi_{1}}$. We conclude
that for each pair of points $(\xi_{1},\xi_{3})\in(\Z^{2})^{2}$ such
that $\xi_{1}\neq\xi_{3}$, there is only one possible choice of the
other two vertices $\left\{ \xi_{2},\xi_{4}\right\} $ such that $Q=(\xi_{1},\xi_{2},\xi_{3},\xi_{4})\in\mc Q_{1,\be}^{0}$,
and similar for $(\xi_{2},\xi_{4})$. By Cauchy-Schwarz inequality,
we have
\begin{align*}
\sum_{Q\in\mc Q_{1,\be}^{0}}f(Q) & =\sum_{Q=(\xi_{1},\xi_{2},\xi_{3},\xi_{4})\in\mc Q_{1,\be}^{0}}f(\xi_{1})f(\xi_{3})\cdot f(\xi_{2})f(\xi_{4})\\
 & \les\sum_{\xi_{1},\xi_{3}\in\Z^{2}}\left(f(\xi_{1})f(\xi_{3})\right)^{2}\\
 & \les\norm f_{\ell^{2}(\Z^{2})}^{4}\les\norm{\la_{j}}_{\ell_{j\le m}^{2}}^{4},
\end{align*}
which is just \eqref{eq:sum Q_ab f(Q) < la_2 ^4} for the case.

\begin{figure}[!h]
\centering
\begin{tikzpicture}[rotate=5]
\centering
\coordinate [label=above left:$\xi_1$]  (1) at (0,2.5);
\coordinate (2) at (0,0);
\coordinate [label=below right:$\xi_3$]  (3) at (4,0);
\coordinate  (4) at (4,2.5);

\coordinate [label=above left:$S_{j_1}$]  (S1) at (-0.75,2.5);
\coordinate  (S2) at (-0.75,0);
\coordinate [label=below right:$S_{j_3}$]  (S3) at (4.75,0);
\coordinate (S4) at (4.75,2.5);
\coordinate [label=left:$\ell_{\xi_1}$] (1-) at (-2,2.5);
\coordinate (4+) at (6,2.5);
\draw (1) -- (2);
\draw (2) -- (3);
\draw (3) -- (4);
\draw (4) -- (1);
\draw (1-) -- (4+);
\foreach \x in {(1), (3)}
{
    \fill \x circle[radius=2pt];
}
\foreach \x in {(1),  (3)}
{
    \draw \x circle[radius=20pt];
}
\end{tikzpicture}
\caption{Determination of a rectangle from given $\xi_1,\xi_3\in\mathbb{Z}^2$}
\end{figure}
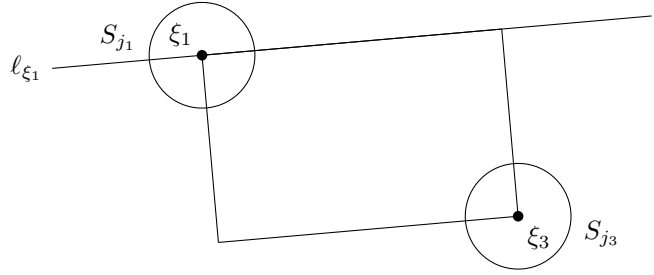

\emph{Case II:} $(\al,\be)=(2,2)$.

Let $j_{1},\ldots,j_{4},a_{1},\ldots,a_{4}$ be integers such that
$0\le j_{k}\le m$ and $1\le a_{k}<j_{k}/2+C$ for $k=1,\ldots,4$.
By \eqref{eq:-2a1+a2+a4}, \eqref{eq:-2a1+a2+a3}, and their cyclic
relabels of indices $1,2,3,4$, for non-negative tuple $\left(c_{k,l}\right)_{k\le4,l\le2}$
such that $\sum_{k=1}^4\sum_{l=1}^2 c_{k,l}=1$, we have
\[
\#\mc Q^{0}(\overrightarrow{j},\overrightarrow{a})  \les2^{\sum_{k=1}^{4}\sum_{l=1}^{2}c_{k,l}\left(2j_{k}-2a_{k}+a_{k+1}+a_{k+1+l}\right)}.
\]
The choices $(c_{k,l})_{k\le4,l\le2}=\frac{1}{24}\cdot((2,3),(3,4),(0,6),(3,3))$
and $\frac{1}{12}\cdot((1,2),(1,2),(3,0),(1,2))$ give
\begin{equation}
\#\mc Q^{0}(\overrightarrow{j},\overrightarrow{a})  \les2^{\frac{1}{2}(j_{1}+j_{2}+j_{3}+j_{4})-\frac{1}{12}(j_{1}-j_{2})},\label{eq:j1-j2}
\end{equation}
\begin{equation}
\#\mc Q^{0}(\overrightarrow{j},\overrightarrow{a})  \les2^{\frac{1}{2}(j_{1}+j_{2}+j_{3}+j_{4})+\frac{1}{6}(a_{1}-a_{2})},\label{eq:a1-a2}
\end{equation}
respectively. Interpolating \eqref{eq:j1-j2}, \eqref{eq:a1-a2}, and their dihedral relabelings of indices $1,2,3,4$,
for $\delta=\frac{1}{10000}$, we have
\[
\#\mc Q^{0}(\overrightarrow{j},\overrightarrow{a})  \les2^{\frac{1}{2}(j_{1}+j_{2}+j_{3}+j_{4})-\delta\sum_{k=1}^{4}\left(\left|j_{k}-j_{k+1}\right|+\left|a_{k}-a_{k+1}\right|\right)},
\]
from which we conclude by 
\[
\mc{Q}^0_{2,2}\subset\bigcup_{\substack{0\le j_k \le m \\ 1\le a_k\le j_k/2+C \\ k=1,2,3,4}}\mc Q^{0}(\overrightarrow{j},\overrightarrow{a}) 
\]
that (using that $f(\xi)=2^{-j/2}\la_j$ for $\xi\in S_j$)
\begin{align*}
\sum_{Q\in\mc Q_{2,2}^{0}}f(Q) & \les\sum_{\substack{0\le j_{k}\le m\\
1\le a_{k}<j_{k}/2+C\\
k=1,2,3,4
}
}\#\mc Q^{0}(\overrightarrow{j},\overrightarrow{a})\cdot  2^{-(j_{1}+j_{2}+j_{3}+j_{4})/2}\la_{j_{1}}\la_{j_{2}}\la_{j_{3}}\la_{j_{4}}\\
 & \les\sum_{\substack{0\le j_{k}\le m\\
1\le a_{k}<j_{k}/2+C\\
k=1,2,3,4
}
}2^{-\delta\sum_{k=1}^{4}\left(\left|j_{k}-j_{k+1}\right|+\left|a_{k}-a_{k+1}\right|\right)}\la_{j_{1}}\la_{j_{2}}\la_{j_{3}}\la_{j_{4}}\\
 & \les\sum_{\substack{0\le j_{k}\le m\\
k=1,2,3,4
}
}\left(2^{-\delta\left|j_{1}-j_{2}\right|/2}\la_{j_{1}}\la_{j_{2}}\cdot2^{-\delta\left|j_{3}-j_{4}\right|}\la_{j_{3}}\la_{j_{4}}\right)\\
 & \cdot\sum_{\substack{\substack{1\le a_{k}\le m/2}
+C\\
k=1,2,3,4
}
}2^{-\delta\left(\left|a_{1}-a_{2}\right|+\left|a_{2}-a_{3}\right|+\left|a_{3}-a_{4}\right|\right)}\\
 & \les\norm{\la_{j}}_{\ell_{j\le m}^{2}}^{4}\cdot\sum_{1\le a_{4}\le m/2+C}1\les m\norm{\la_{j}}_{\ell_{j\le m}^{2}}^{4},
\end{align*}
which is just \eqref{eq:sum Q_22 f(Q) < m*la_2^4}.

We pass to showing \eqref{eq:sum Q_22 1/gcd f(Q) < la_2^4}, which
is just a repeat of the preceding proof. By \eqref{eq:-2a1+a2+a4},
\eqref{eq:-2a1+a2+a3}, \eqref{eq:-2a1+a2+a4/2}, and their cyclic
relabels, for non-negative tuple $\left(c_{k,l}\right)_{k\le4,l\le2}$
such that $\sum_{k,l}c_{k,l}=1$, we have
\[
\sum_{(\xi_{1},\xi_{2},\xi_{3},\xi_{4})\in\mc Q^{0}(\overrightarrow{j},\overrightarrow{a})  }\frac{1}{\gcd(\xi_{1}-\xi_{4})}\les2^{\sum_{k=1}^{4}\sum_{l=1}^{2}c_{k,l}\left(2j_{k}-2a_{k}+a_{k+1}+a_{k+1+l}\right)}\cdot2^{-c_{1,2}\cdot a_{4}/2}.
\]
Plugging the same choices of $\left(c_{k,l}\right)_{k\le4,l\le2}$,
we obtain
\[
\sum_{(\xi_{1},\xi_{2},\xi_{3},\xi_{4})\in\mc Q^{0}(\overrightarrow{j},\overrightarrow{a})  }\frac{1}{\gcd(\xi_{1}-\xi_{4})}\les2^{\frac{1}{2}(j_{1}+j_{2}+j_{3}+j_{4})-\delta\sum_{k=1}^{4}\left(\left|j_{k}-j_{k+1}\right|+\left|\al_{k}-\al_{k+1}\right|\right)}\cdot2^{-\delta a_{4}},
\]
concluding that
\begin{align*}
&\sum_{Q=(\xi_1,\xi_2,\xi_3,\xi_4)\in\mc Q_{2,2}^{0}}\frac{1}{\gcd(\xi_1-\xi_4)}f(Q)\\
& \les\sum_{\substack{0\le j_{k}\le m\\
1\le a_{k}<j_{k}/2+C\\
k=1,2,3,4
}
}\sum_{(\xi_{1},\xi_{2},\xi_{3},\xi_{4})\in\mc Q^{0}(\overrightarrow{j},\overrightarrow{a})  }\frac{1}{\gcd(\xi_{1}-\xi_{4})}2^{-(j_{1}+j_{2}+j_{3}+j_{4})/2}\la_{j_{1}}\la_{j_{2}}\la_{j_{3}}\la_{j_{4}}\\
 & \les\sum_{\substack{0\le j_{k}\le m\\
1\le a_{k}<j_{k}/2+C\\
k=1,2,3,4
}
}2^{-\delta\sum_{k=1}^{4}\left(\left|j_{k}-j_{k+1}\right|+\left|a_{k}-a_{k+1}\right|\right)}\la_{j_{1}}\la_{j_{2}}\la_{j_{3}}\la_{j_{4}}\cdot2^{-\delta a_{4}}\\
 & \les\sum_{\substack{0\le j_{k}\le m\\
k=1,2,3,4
}
}\left(2^{-\delta\left|j_{1}-j_{2}\right|/2}\la_{j_{1}}\la_{j_{2}}\cdot2^{-\delta\left|j_{3}-j_{4}\right|}\la_{j_{3}}\la_{j_{4}}\right)\\
 & \cdot\sum_{\substack{\substack{1\le a_{k}\le m/2}
+C\\
k=1,2,3,4
}
}2^{-\delta\left(\left|a_{1}-a_{2}\right|+\left|a_{2}-a_{3}\right|+\left|a_{3}-a_{4}\right|\right)}\cdot2^{-\delta a_{4}}\\
 & \les\norm{\la_{j}}_{\ell_{j\le m}^{2}}^{4},
\end{align*}

which is just \eqref{eq:sum Q_22 1/gcd f(Q) < la_2^4}.

\emph{Case III:} $(\al,\be)=(3,3)$.

For $j_{1},j_{2},j_{3},j_{4}\in\N$, by Lemma \ref{lem:j1+j2+a3}, we have
\[
q_{j_{1},j_{2},j_{3},j_{4}}:=\#\mc Q_{3,3}^{0}\cap\left(S_{j_{1}}\times S_{j_{2}}\times S_{j_{3}}\times S_{j_{4}}\right)\les\min_{k=1,2,3,4}2^{j_{k}+j_{k+1}}.
\]
One can check
\[
\min_{k=1,2,3,4}\{j_{k}+j_{k+1}\}-\frac{1}{2}\left(j_{1}+j_{2}+j_{3}+j_{4}\right)\le-\frac{1}{100}\left(\left|j_{1}-j_{3}\right|+\left|j_{2}-j_{4}\right|\right),
\]
and so
\begin{align*}
\sum_{Q\in\mc Q_{3,3}^{0}}f(Q) & \les\sum_{j_{1},j_{2},j_{3},j_{4}\ge0}q_{j_{1},j_{2},j_{3},j_{4}}2^{-\frac{1}{2}(j_{1}+j_{2}+j_{3}+j_{4})}\la_{j_{1}}\la_{j_{2}}\la_{j_{3}}\la_{j_{4}}\\
 & \les\sum_{j_{1},j_{3}\ge0}2^{-\frac{1}{100}\left|j_{1}-j_{3}\right|}\la_{j_{1}}\la_{j_{3}}\cdot\sum_{j_{2},j_{4}\ge0}2^{-\frac{1}{100}\left|j_{2}-j_{4}\right|}\la_{j_{2}}\la_{j_{4}}\\
 & \les\norm{\la_{j}}_{\ell_{j\le m}^{2}}^{4},
\end{align*}
which is just \eqref{eq:sum Q_ab f(Q) < la_2 ^4} for the case.

\emph{Case IV:} $(\al,\be)=(2,3)$ (or $(3,2)$).

For $j_{1},j_{2},j_{3},j_{4}\in\N$, by Lemma \ref{lem:j1+j2+a3} we have
\begin{equation}
q_{j_{1},j_{2},j_{3},j_{4}}:=\#\mc Q_{2,3}^{0}\cap\left(S_{j_{1}}\times S_{j_{2}}\times S_{j_{3}}\times S_{j_{4}}\right)\les2^{\min\left\{ j_{1}+j_{4},j_{2}+j_{3},j_{1}+j_{2}\right\} }.\label{eq:(3,2) 1}
\end{equation}

For $\overrightarrow{a}=(a_1,a_2,0,0)$ with integers $a_{1},a_{2}$ such that $1\le a_{1}<j_{1}/2+C$ and
$1\le a_{2}<j_{2}/2+C$, by Lemma \ref{lem:j1+j2+a3} and \eqref{eq:-2a1+a2+a3}, we also have
\begin{equation}\#\mc Q^{0}(\overrightarrow{j},\overrightarrow{a}) \les 2^{j_3+j_4+\min\{ a_1,a_2 \}} \les 2^{j_{3}+j_{4}+\frac{1}{2}(a_{1}+a_{2})}\label{eq:(3,2) 2-1}\end{equation}
and
\begin{align}
\#\mc Q^{0}(\overrightarrow{j},\overrightarrow{a})  & \les\min\left\{ 2^{2j_{1}-2a_{1}},2^{2j_{2}-2a_{2}}\right\} \label{eq:(3,2) 2-2}\\
 & \les2^{j_{1}+j_{2}-(a_{1}+a_{2})}.\nonumber 
\end{align}

Interpolating \eqref{eq:(3,2) 2-1} and \eqref{eq:(3,2) 2-2}, we
have
\begin{align*}
\#\mc Q^{0}(\overrightarrow{j},\overrightarrow{a})  & \les2^{\frac{3}{5}\left(j_{3}+j_{4}+\frac{1}{2}(a_{1}+a_{2})\right)+\frac{2}{5}\left(j_{1}+j_{2}-(a_{1}+a_{2})\right)}\\
&=2^{\frac{3}{5}\left(j_{3}+j_{4}\right)+\frac{2}{5}\left(j_{1}+j_{2}\right)-\frac{1}{10}\left(a_{1}+a_{2}\right)},
\end{align*}
which implies
\begin{align}
q_{j_{1},j_{2},j_{3},j_{4}} & \le\sum_{\substack{1\le a_{1}<j_{1}/2+C\\
1\le a_{2}<j_{2}/2+C
}
}\#\mc Q^{0}(\overrightarrow{j},\overrightarrow{a}) \label{eq:(3,2) 2}\\
 & \les\sum_{a_{1},a_{2}\in\N}2^{\frac{3}{5}\left(j_{3}+j_{4}\right)+\frac{2}{5}\left(j_{1}+j_{2}\right)-\frac{1}{10}\left(a_{1}+a_{2}\right)}\nonumber \\
 & \les2^{\frac{3}{5}\left(j_{3}+j_{4}\right)+\frac{2}{5}\left(j_{1}+j_{2}\right)}.\nonumber 
\end{align}
By \eqref{eq:(3,2) 1}, \eqref{eq:(3,2) 2}, and the inequality
\begin{align*}
 & \min\left\{ j_{1}+j_{4},j_{2}+j_{3},j_{1}+j_{2},\frac{3}{5}\left(j_{3}+j_{4}\right)+\frac{2}{5}\left(j_{1}+j_{2}\right)\right\} -(j_{1}+j_{2}+j_{3}+j_{4})/2\\
 & \le-\frac{1}{100}\left(\left|j_{1}-j_{3}\right|+\left|j_{2}-j_{4}\right|\right),
\end{align*}
we conclude
\begin{align*}
\sum_{Q\in\mc Q_{2,3}^{0}}f(Q) & \les\sum_{j_{1},j_{2},j_{3},j_{4}\ge0}q_{j_{1},j_{2},j_{3},j_{4}}2^{-(j_{1}+j_{2}+j_{3}+j_{4})/2}\la_{j_{1}}\la_{j_{2}}\la_{j_{3}}\la_{j_{4}}\\
 & \les\sum_{j_{1},j_{2},j_{3},j_{4}\ge0}2^{-\frac{1}{100}\left(\left|j_{1}-j_{3}\right|+\left|j_{2}-j_{4}\right|\right)}\la_{j_{1}}\la_{j_{2}}\la_{j_{3}}\la_{j_{4}}\\
 & \les\sum_{j_{1},j_{3}\ge0}2^{-\frac{1}{100}\left|j_{1}-j_{3}\right|}\la_{j_{1}}\la_{j_{3}}\cdot\sum_{j_{2},j_{4}\ge0}2^{-\frac{1}{100}\left|j_{2}-j_{4}\right|}\la_{j_{2}}\la_{j_{4}}\\
 & \les\norm{\la_{j}}_{\ell_{j\le m}^{2}}^{4},
\end{align*}
which is just \eqref{eq:sum Q_ab f(Q) < la_2 ^4} for the case.
\end{proof}
\begin{rem} We thank Po-Lam Yung for the following more conceptional explanation of above interpolation type arguments: For example in Case IV,  $(\frac12,\frac12,\frac12,\frac12)$ is in the interior of the convex hull $C$ of $(1,0,0,1)$, $(0,1,1,0)$, $(1,1,0,0)$ and $(\frac25,\frac25,\frac35,\frac35)$. More precisely
$$
(\tfrac12,\tfrac12,\tfrac12,\tfrac12) = \tfrac15(1,0,0,1) + \tfrac15(0,1,1,0) + \tfrac{1}{10}(1,1,0,0) + \tfrac12(\tfrac25,\tfrac25,\tfrac35,\tfrac35).$$

All these points lie in the plane $P=\{x_1+x_3 = 1 \text{ and } x_2+x_4 = 1\}$. Hence, for small $\delta > 0$, the four points
$$(\tfrac12\pm_1 \delta, \tfrac12\pm_2\delta, \tfrac12\mp_1\delta, \tfrac12\mp_2\delta), \quad \pm_1,\pm_2\in \{-,+\}$$
are all in $P\cap C$.
Therefore, regardless of the signs of $j_1-j_3$ and $j_2-j_4$, there exist $c_j\geq 0$ satisfying $c_1+c_2+c_3+c_4 = 1$ so that
\begin{align*}
    &c_1(j_1+j_4) + c_2(j_2+j_3) + c_3(j_1+j_2) + c_4( \tfrac35(j_3+j_4)+\tfrac25(j_1+j_2) ) \\
    ={}& \tfrac12(j_1+j_2+j_3+j_4) - \delta (|j_1-j_3| + |j_2-j_4|),
\end{align*}
and in the argument in Case IV above we have chosen $\delta = \tfrac1{100}$. 
\end{rem}
This completes the overall proof of Theorem \ref{thm:local Stri}.

\section{Proof of Theorem \ref{thm:GWP}}\label{sec:proof-gwp}
We only carry out the proof on the relevant case $0<s\le 1$, which is most convenient with adapted function spaces. For this purpose, we recall the definition of the function space $Y^{s}$ from \cite{herr2011global} and relevant facts.
For a general theory, we refer to \cite{koch2014dispersive,herr2011global,hadac2009well,hadac2010erratum}.
\begin{defn}

Let $\mathcal{Z}$ be the collection of finite non-decreasing sequences
$\left\{ t_{k}\right\} _{k=0}^{K}$ in $\R$. We define $V^{2}$ as
the space of all right-continuous functions $u:\R\rightarrow\C$ with
$\lim_{t\rightarrow-\infty}u(t)=0$ and
\[
\|u\|_{V^{2}}:=\left(\sup_{\left\{ t_{k}\right\} _{k=0}^{K}\in\mc Z}\sum_{k=1}^{K}\left|u(t_{k})-u(t_{k-1})\right|^{2}\right)^{1/2}<\infty.
\]
For $s\in\R$, we define $Y^{s}$ as the space of $u:\R\times\T^{2}\rightarrow\C$
such that $e^{it|\xi|^2}\widehat{u(t)}(\xi)$ lies in $V^{2}$ for each $\xi\in\Z^{2}$
and
\[
\|u\|_{Y^{s}}\:=\left(\sum_{\xi\in\Z^{2}}\left(1+\left|\xi\right|^{2}\right)^{s}\|e^{it|\xi|^{2}}\widehat{u(t)}(\xi)\|_{V^{2}}^{2}\right)^{1/2}<\infty.
\]
\end{defn}
For time interval $I\subset\R$, we also consider the restriction space $Y^s(I)$ of $Y^s$.

The space $Y^{s}$ is used in \cite{herr2011global} and later works on critical regularity theory of Schr{\"o}dinger equations on periodic domains. Some well-known properties are the following.
\begin{prop}[{\cite[Section 2]{herr2011global}}]
$Y^{s}$-norms have the following properties.
\begin{itemize}
\item Let $A,B$ be disjoint subsets of $\Z^{2}$. For $s\in\R$, we have
\begin{equation}
\|P_{A\cup B}u\|_{Y^{s}}^{2}=\|P_{A}u\|_{Y^{s}}^{2}+\|P_{B}u\|_{Y^{s}}^{2}.\label{eq:l^2_=00005Cxistructure}
\end{equation}
\item For $s\in\R$, time $T>0$, and a function $f\in L^{1}H^{s}$, denoting
\[
\mc I(f)(t):=\int_{0}^{t}e^{i(t-t')\De}f(t')dt'
\]
we have 
\begin{equation}
\norm{\chi_{[0,T)}\cdot\mc I(f)}_{Y^{s}}\les\sup_{v\in Y^{-s}:\norm v_{Y^{-s}}\le1}\left|\int_0^T\int_{\T^{2}}f\overline{v}dxdt\right|.\label{eq:U2V2}
\end{equation}
\item For time $T>0$ and a function $\phi\in H^s(\T^2)$, we have
\begin{equation}\label{eq:free}
\norm{\chi_{[0,T)}\cdot e^{it\De}\phi}_{Y^s}\approx\norm{\phi}_{H^s}
\end{equation}
and for function $u\in Y^s$, $u\in L^\infty H^s$ and
\begin{equation}\label{eq:Y^s>L infty}
\norm{\chi_{[0,T)}u}_{Y^s}\gtrsim\norm{u}_{L^\infty([0,T);H^s)}.
\end{equation}
\end{itemize}
\end{prop}
For $N\in2^{\N}$, denote by $\mc C_{N}$ the set of cubes of size
$N$
\[
\mc C_{N}:=\left\{ (0,N]^{2}+N\xi_{0}:\xi_{0}\in\Z^{2}\right\} .
\]
We transfer \eqref{eq:localStri} to the following estimate.
\begin{lem}\label{lem:L4}
For all $N\in2^{\N}$, intervals $I\subset\R$ such
that $\left|I\right|\le\frac{1}{\log N}$, cubes $C\in\mc C_{N}$,
and $u\in Y^{0}$, we have
\begin{equation}
\norm{P_{C}u}_{L^{4}_{t,x}(I\times\T^{2})}\les\norm u_{Y^{0}}.\label{eq:local Y^0 Stri}
\end{equation}
\end{lem}

\begin{proof}
We follow the notations in \cite[Section 2]{herr2011global}. Let
$u$ be a $U_{\De}^{4}L^2$-atom, i.e.
\[
u(t)=\sum_{j=1}^{J}1_{[t_{j-1},t_{j})}e^{it\De}\phi_{j}
\]
for $\phi_{1},\ldots,\phi_{J}\in L^{2}(\T^{2})$, $t_0\le...\le t_J$, $\sum_{j=1}^{J}\norm{\phi_{j}}_{L^{2}}^{4}=1$.
By \eqref{eq:localStri}, we have
\begin{equation}
\norm{P_{C}u}_{L^{4}_{t,x}(I\times\T^{2})}^{4}\les\sum_{j=1}^{J}\norm{P_{C}e^{it\De}\phi_{j}}_{L^{4}_{t,x}(I\times\T^{2})}^{4}\les\sum_{j=1}^{J}\norm{\phi_{j}}_{L^{2}}^{4}\les 1.\label{eq:local Y^0 Stri J}
\end{equation}
By \cite[Proposition 2.3]{herr2011global} and \eqref{eq:local Y^0 Stri J},
for $u\in Y^{0}$ we conclude
\[
\norm{P_{C}u}_{L_{t,x}^{4}(I\times\T^{2})}\les\norm u_{U_{\De}^{4}L^2}\les\norm u_{V_{\De}^{2}L^2}\les\norm u_{Y^{0}}.\qedhere
\]
\end{proof}
Since we only rely on the $L^4$ estimate, Lemma \ref{lem:L4} explains why we can work with the $Y^s$-norm, instead of the $U^2$-based space as was used in \cite{herr2011global}. 

For $N\in2^\N$, we set the interval $I_N:=[0,1/\log N)$. Let $Z_N$ be the norm
\[
\norm{u}_{Z_N}:=\norm{\chi_{I_N}\cdot u}_{Y^0}+N^{-s}\norm{\chi_{I_N}\cdot u}_{Y^s}.
\]
We show our main trilinear estimate:

\begin{lem}\label{lem:tri-est}
For $0<s\le 1$ and $N\gg_s 1$, we have
\begin{equation}\label{eq:trilinear}
\norm{\mc I(u_1u_2u_3)}_{Z_N}\les\norm{u_1}_{Z_N}\norm{u_2}_{Z_N}\norm{u_3}_{Z_N},
\end{equation}
where each $u_j$ could also be replaced by its complex conjugate. The implicit constant is independent from $s$.
\end{lem}
\begin{proof}
Let $k_s\approx 1/s$ be an integer. In this proof we use $2^{k_s}$-adic cutoffs: for $N\in 2^{k_s\N}$, we denote
\[
P_{\sim N}u=u_{\sim N}=u_{<2^{k_s}N}-u_{<N}.
\]
Since $\norm{\chi_{I_N}\cdot u}_{Z_{\tilde N}}\approx\norm{u}_{Z_N}$ holds for $\tilde N\in[2^{-k_s}N,N]$, we assume further that $N\in 2^{k_s\N}$.
\eqref{eq:trilinear} is reduced to showing
\begin{equation}\label{eq:bootstrap1}
\left|\int_{I_N\times\T^2}u_1u_2u_3\cdot v_{<N}dxdt\right|\les\norm{u_1}_{Z_N}\norm{u_2}_{Z_N}\norm{u_3}_{Z_N}\norm v_{Y^{0}}
\end{equation}
and
\begin{equation}\label{eq:bootstrap2}
\left|\int_{I_N\times\T^2}u_1u_2u_3\cdot v_{\ge N}dxdt\right|\les\norm{u_1}_{Z_N}\norm{u_2}_{Z_N}\norm{u_3}_{Z_N}\cdot N^{s}\norm v_{Y^{-s}}
\end{equation}
with implicit constants in \eqref{eq:bootstrap1} and \eqref{eq:bootstrap2} independent from $s$.

We prove \eqref{eq:bootstrap1} and \eqref{eq:bootstrap2}.
For $M\ge N$ in $2^{k_s\N}$ and $C\in\mc{C}_M$, partitioning $I_N$
to intervals of length comparable to $\frac{1}{\log M}$ and applying
\eqref{eq:local Y^0 Stri} to each, we have
\begin{equation}
\norm{\chi_{I_N}\cdot P_C u}_{L_{t,x}^{4}}\les\left(\frac{\log M}{\log N}\right)^{1/4}\norm u_{Y^{0}}.\label{eq:linear Y^0_M}
\end{equation}

By \eqref{eq:linear Y^0_M}, for $u\in Y^{s}$ we have
\begin{align}\label{eq:L^4 stri}
\norm{\chi_{I_N}\cdot u}_{L_{t,x}^{4}} & \les\norm{u_{<N}}_{Y^{0}}+\sum_{M\ge N}\left(\frac{\log M}{\log N}\right)^{1/4}\norm{u_{\sim M}}_{Y^{0}}\\
 & \les\norm u_{Y^{0}}+\sum_{M\ge N}\left(\frac{\log M}{\log N}\right)^{1/4}\frac{N^{s}}{M^{s}}\cdot N^{-s}\norm u_{Y^{s}}\nonumber\\
 & \les\norm u_{Y^{0}}+N^{-s}\norm u_{Y^{s}}\les\norm{u}_{Z_N},\nonumber
\end{align}
which implies \eqref{eq:bootstrap1}.

We prove \eqref{eq:bootstrap2} by partitioning the frequency domain
$\Z^{2}$ into congruent cubes. By \eqref{eq:linear Y^0_M} and \eqref{eq:l^2_=00005Cxistructure}, for $M\in2^{k_s\N}$ and $u,v\in Y^{0}$ we have
\begin{align}
  &\norm{\chi_{I_N} \cdot P_{\le M}\left(uv\right)}_{L_{t,x}^{2}}\label{eq:bilinear}\\
 & \les\sum_{\substack{C_{1},C_{2}\in\mc C_{M}\\
  \dist(C_{1},C_{2})\le M
}
}\norm{\chi_{I_N}\cdot P_{C_{1}}u\cdot P_{C_{2}}v}_{L_{t,x}^{2}}\nonumber \\
 & \les\sum_{\substack{C_{1},C_{2}\in\mc C_{M}\\
\dist(C_{1},C_{2})\le M
}
}\norm{\chi_{I_N}\cdot P_{C_{1}}u}_{L_{t,x}^{4}}\norm{\chi_{I_N}\cdot P_{C_{2}}v}_{L_{t,x}^{4}}\nonumber \\
 & \les\left(1+\frac{\log M}{\log N}\right)^{1/2}\left(\sum_{C\in\mc C_{M}}\norm{P_{C}u}_{Y^{0}}^2\sum_{C\in\mc C_{M}}\norm{P_{C}v}_{Y^{0}}^2\right)^{1/2}\nonumber \\
 & \les\left(1+\frac{\log M}{\log N}\right)^{1/2}\norm u_{Y^{0}}\norm v_{Y^{0}}.\nonumber 
\end{align}
We conclude quadrilinear estimates. By \eqref{eq:bilinear} and Young's convolution inequality on $(L,K)$ using that $\sum_{R\in2^{k_s\N}}R^{-s}\les 1$, we have
\begin{align}
 & \sum_{K\ge N}\sum_{L\gtrsim K}\left|\int_{I_N\times\T^{2}}P_{<N}(u_1u_2)P_{<N}(w_{\sim L}v_{\sim K})dxdt\right|\label{eq:quar1}\\
 & \les\norm {u_1}_{Y^{0}}\norm {u_2}_{Y^{0}}\sum_{K\ge N}\sum_{L\gtrsim K}\norm{w_{\sim L}}_{Y^{0}}\norm{v_{\sim K}}_{Y^{0}}\nonumber \\
 & \les\norm {u_1}_{Y^{0}}\norm {u_2}_{Y^{0}}\sum_{K\ge N}\sum_{L\gtrsim K}(L/K)^{-s}\norm{w_{\sim L}}_{Y^{s}}\norm{v_{\sim K}}_{Y^{-s}}\nonumber \\
 & \les\norm {u_1}_{Y^{0}}\norm {u_2}_{Y^{0}} \norm w_{Y^{s}} \norm v_{Y^{-s}}\nonumber 
\end{align}
and
\begin{align}
 & \sum_{M\ge N}\sum_{K\ge N}\sum_{L\gtrsim K}\left|\int_{I\times\T^{2}}P_{\sim M}\left(u_1u_2\right)P_{\sim M}\left(w_{\sim L}v_{\sim K}\right)dxdt\right|\label{eq:quar2}\\
 & \les\sum_{M\ge N}\frac{\log M}{\log N}(\norm{P_{\ge M/4}u_1}_{Y^0}\norm{u_2}_{Y^0}+\norm{u_1}_{Y^0}\norm{P_{\ge M/4}u_2}_{Y^0})\nonumber \\
 & \cdot\sum_{K\ge N}\sum_{L\gtrsim K}\norm{w_{\sim L}}_{Y^{0}}\norm{v_{\sim K}}_{Y^{0}}\nonumber \\
 & \les\sum_{M\ge N}\frac{\log M}{\log N}\frac{N^s}{M^s}\norm {u_1}_{Z_N} \norm {u_2}_{Z_N}\sum_{K\ge N}\sum_{L\gtrsim K}\norm{w_{\sim L}}_{Y^{0}}\norm{v_{\sim K}}_{Y^{0}}\nonumber \\
 & \les\norm {u_1}_{Z_N} \norm {u_2}_{Z_N}\sum_{K\ge N}\sum_{L\gtrsim K}(L/K)^{-s}\norm{w_{\sim L}}_{Y^s}\norm{v_{\sim K}}_{Y^{-s}}\nonumber \\
 & \les \norm {u_1}_{Z_N} \norm {u_2}_{Z_N} \norm w_{Y^s} \norm v_{Y^{-s}}.\nonumber 
\end{align}

Combining \eqref{eq:quar1} and \eqref{eq:quar2}, we have
\begin{equation}
\sum_{K\ge N}\sum_{L\gtrsim K}\left|\int_{I_N\times\T^{2}}(u_1u_2)w_{\sim L}v_{\sim K}dxdt\right|\les\norm{u_1}_{Z_N}\norm{u_2}_{Z_N}\norm{w}_{Z_N}N^{s}\norm v_{Y^{-s}}.\label{eq:quar 1+2}
\end{equation}
Note that in \eqref{eq:bilinear}, \eqref{eq:quar1}, \eqref{eq:quar2}, \eqref{eq:quar 1+2} each function on the left hand side could be replaced by its complex conjugate. We bound
\begin{align*} 
& \left|\int_{I_N\times\T^{2}}u_1u_2u_3v_{\ge N}dxdt\right|
\\
& \le\sum_{K\ge N}\left|\int_{I_N\times\T^{2}}P_{\ge K/4}u_1\cdot u_2\cdot u_3\cdot  v_{\sim K}dxdt\right|
\\
& +\sum_{K\ge N}\left|\int_{I_N\times\T^{2}}P_{<K/4}u_1\cdot P_{\ge K/4}u_2\cdot u_3\cdot v_{\sim K}dxdt\right|
\\
& +\sum_{K\ge N}\left|\int_{I_N\times\T^{2}}P_{<K/4}u_1\cdot P_{<K/4}u_2\cdot P_{\ge K/4}u_3\cdot v_{\sim K}dxdt\right|,
\end{align*}
applying \eqref{eq:quar 1+2} to each term we conclude \eqref{eq:bootstrap2}.
\end{proof}

\begin{proof}[Proof of Theorem \ref{thm:GWP}]
Let $s>0$ and $N\gg_s 1$. By \eqref{eq:trilinear}, \eqref{eq:U2V2}, and the expansion $|u|^2u-|v|^2v=(|u|^2+\overline{u}v)(u-v)+v^2\overline{(u-v)}$ we have
\begin{equation}\label{eq:difference}
\norm{\mc I(|u|^2u-|v|^2v)}_{Z_N}\les(\norm{u}_{Z_N}+\norm{v}_{Z_N})^2\norm{u-v}_{Z_N}.
\end{equation}
Based on \eqref{eq:difference}, we use the contraction mapping principle.
Let $B_N\subset H^s$ be the ball
\[
B_N:=\{u_0\in H^s:\norm{u_0}_{L^2}+N^{-s}\norm{u_0}_{H^s}\le2\delta\},
\]
and $X_N$ be the complete metric space
\[
X_N:=\{u\in C^0(I_N;H^s)\cap Y^s(I_N):\norm{u}_{Z_N}\le\eta\}
\]
equipped with the norm $Z_N$, where $\delta,\eta>0$ are universal constants to be fixed shortly.

By \eqref{eq:difference}, there exists $\eta>0$ such that the map
\[
u\mapsto\mc I(|u|^2u)
\]
is a contraction map on $X_N$ of Lipschitz constant $1/2$, which fixes $0$.

By \eqref{eq:free}, there exists $\delta>0$ such that 
\begin{equation}\label{eq:free<eta}
\norm{e^{it\De}\phi}_{Z_N}<\eta/4
\end{equation}
holds for every $\phi\in B_N$, so that the map
\[
u\mapsto e^{it\De}u_0\mp i \mc I(|u|^2u)
\]
is a contraction mapping on $X_N$. Thus for $u_0\in B_N$, there exists a solution $u$ to \eqref{eq:NLS} in $X_N$ on time interval $I_N$. Moreover, since the map $u\mapsto \mc I(|u|^2u)$ is a contraction map of Lipschitz constant $1/2$, given solutions $u,v\in X_N$ to $u_0,v_0\in B_N$, we have
\begin{align*}
\norm{u-v}_{Z_N}&\le\norm{e^{it\De}(u_0-v_0)}_{Z_N}+\norm{\mc I(|u|^2u)-\mc I(|v|^2v)}_{Z_N}\\
&\le\norm{e^{it\De}(u_0-v_0)}_{Z_N}+\frac12\norm{u-v}_{Z_N},
\end{align*}
which implies that the flow map $u_0\mapsto u\in X_N$ is Lipschitz continuous by \eqref{eq:free}.

We then check uniqueness. Let $u,v\in Y^s\cap C^0H^s$ be solutions to \eqref{eq:NLS} on a time interval $[0,T),T>0$, with common initial data $u_0$ such that $\norm{u_0}_{L^2}\le\delta$. There exists $N_0\gg_s 1$ such that $I_{N_0}\subset[0,T)$ and
\[
\norm{u_{>N_0}}_{Y^0}+N_0^{-s}\norm{u}_{Y^s}\le 2N_0^{-s}\norm{u}_{Y^s}\le\eta/2,
\]
\[
\norm{v_{>N_0}}_{Y^0}+N_0^{-s}\norm{v}_{Y^s}\le 2N_0^{-s}\norm{v}_{Y^s}\le\eta/2.
\]
We have
\begin{align*}
\norm{P_{\le N_0}(u-e^{it\De}u_0)}_{Y^0(I_N)}&\les\norm{P_{\le N_0}(|u|^2u)}_{L^1(I_N;L^2)}\\
&\les N_0\norm{|u|^2u}_{L^1(I_N;L^1)}\les N_0\norm{u}_{L^4(I_N;L^4)}^3,
\end{align*}
which shrinks to zero as $N\rightarrow\infty$ since $u\in L^4_{t,x}$ on $I_{N_0}$ by \eqref{eq:L^4 stri}. Thus, applying the same argument to $v$, by \eqref{eq:free<eta} there exists $N\ge N_0$ such that
\[
\chi_{I_N}u,\chi_{I_N}v\in X_N,
\]
which implies $u=v$ on $I_N$. Therefore, the maximal time $t_*\ge 0$ that $u= v$ on $[0,t_*]$ cannot be less than $T$, implying the uniqueness of solution to \eqref{eq:NLS}.

In summary, we proved uniform Lipschitz local well-posedness of \eqref{eq:NLS} mapping $B_N$ to $X_N$. It remains to show extend the lifespan over arbitrarily large time interval. For $N\gg_s 1,t_0\in\R$, and a solution $u\in Y^s$ to \eqref{eq:NLS} such that $u(t_0)\in B_N$ and $\norm{u(t_0)}_{L^2}\le\delta$, by \eqref{eq:Y^s>L infty} we have
\[
N^{-s}\norm{u(t_0+\frac1{2\log N})}_{H^s}\les\norm{u}_{Z_N}\le\eta.
\]
Moreover, since $u(t_0)$ is a limit of smooth data in $B_N$ and solutions to \eqref{eq:NLS} in $C^0H^2$ conserve their $L^2$-norms, we have
\[
\norm{u(t_0+\frac1{2\log N})}_{L^2}=\norm{u(t_0)}_{L^2}\le\delta.
\]
Thus, there exists a constant $K\in2^\N$ such that $u(t_0+\frac{1}{2\log N})\in B_{KN}$.

Let $u_0\in H^s$ be any function that $\norm{u_0}_{L^2}\le\delta$. Let $N_0\gg_s1$ be a dyadic number such that $u_0\in B_{N_0}$. For $j\in\N$, let
\[
N_j:=K^jN_0\text{ and }T_j:=\sum_{k=0}^{j-1}\frac{1}{2\log N_k}.
\]
We extend the solution inductively. For $j\in\N$, we can extend the solution $u\in Y^s$ to \eqref{eq:NLS} on $[0,T_j]$ to $[0,T_{j+1}]$ with $u(T_{j+1})\in B_{N_{j+1}}$. Since $\lim_{j\rightarrow\infty}T_j=\infty$ the lifespan of $u$ is infinite.
\end{proof}
\section*{Acknowledgement}
The authors are grateful to Ciprian Demeter for kindly pointing out an error in the first preprint version of this paper. In addition, the authors thank Ciprian Demeter and Po-Lam Yung for carefully reading Sections 1-3 and a number of remarks which helped us to improve the exposition. Also, we thank the anonymous referees, in particular for their detailed suggestions concerning the exposition in Section 4.

Funded by the Deutsche Forschungsgemeinschaft (DFG, German Research Foundation) -- IRTG 2235 -- Project-ID 282638148
\bibliographystyle{amsplain}
\bibliography{citationforTd}
\end{document}